\documentclass[11pt,a4paper,english]{article}
\usepackage[utf8]{inputenc}
\usepackage[T1]{fontenc}
\usepackage{babel}
\usepackage{amsmath}
\usepackage{amssymb}
\usepackage{amsthm}
\usepackage{graphicx}
\usepackage{times} 
\usepackage{mathtools}
\usepackage{geometry}
\usepackage{yhmath}
\usepackage{sectsty}
\usepackage{mathdots}
\usepackage{float}
\usepackage{pgfplots}
\usepackage{tikz}
\usepackage{comment}
\usepackage{pdfpages}
\usepackage[bookmarks=true,bookmarksopen=true,colorlinks=true,linkcolor=blue,citecolor=blue,urlcolor=blue]{hyperref}
\usepackage[textwidth=1.7cm, backgroundcolor=blue!10]{todonotes}
\usepackage{xcolor}

\sectionfont{\centering}
\subsectionfont{\centering}

\title{Geometric Invariants of Plane and Space Curves}
\author{Hana Mel\'{a}nov\'{a}\footnote{ Supported by Austrian Academy of Sciences, ÖAW, Doc Stipendium FA506081, and by Austrian Science Fund, FWF, Project AP31338} \\
   Faculty of Mathematics, University of Vienna, Austria\\
   \href{mailto:hana.melanova@univie.ac.at}{hana.melanova@univie.ac.at}}

\DeclareMathOperator{\codim}{codim}
\DeclareMathOperator{\ord}{ord}
\DeclareMathOperator{\Quot}{Quot}
\DeclareMathOperator{\rk}{rk}

\DeclareMathOperator{\Aut}{Aut}

\DeclareMathOperator{\GL}{GL}

\newcommand{\ps}[1]{[\![#1]\!]}

\newtheorem{remark}{Remark}[section]
\newtheorem{theorem}[remark]{Theorem}
\newtheorem{definition}[remark]{Definition}
\newtheorem{lemma}[remark]{Lemma}
\newtheorem{proposition}[remark]{Proposition}
\newtheorem{corollary}[remark]{Corollary}
\newtheorem{example}[remark]{Example}

\begin{document}

\maketitle
\begin{abstract}
The traditional study of plane and space algebraic curves by looking at their tangent vectors, curvatures and torsions provides geometric, but unfortunately not sufficient information about individual curves in order to be able to distinguish between two different ones. The aim of this note is to generalize the classical concept of curvature and torsion to \emph{higher algebraic curvatures}, a family of generators of so-called \emph{geometric invariants} --- algebraic quantities of parametrized curves that are equivariant under reparametrizations. We show, further, that each such quantity can be equally expressed as a rational function in the defining implicit equations of the curve and their higher partial derivatives. Moreover, we prove that each smooth analytic branch of a complex algebraic curve is already uniquely determined by the algebraic curvatures. 
\end{abstract}

\section{Introduction and Geometric Motivation}

In geometry, in order to study and understand the local shape of a plane curve in $\mathbb{A}_{\mathbb{R}}^2$, one usually looks at its tangent line and curvature at a particular point. In the case of a space curve in~$\mathbb{A}_{\mathbb{R}}^3$, one has also the notion of torsion which measures the deviation of the curve from being a plane one. All these numerals reflect several local geometric properties of the curve at a given point. They, however, do not determine the curve uniquely, even not its analytic branches at a given point. 

Let us consider the circle centered at the point $(0,-\frac{1}{2})$ with radius $\frac{1}{2}$, defined by the equation 
\[
f(x,y)=x^2+\left(y+\frac{1}{2}\right)^2-\frac{1}{4},
\]
and the node defined by
\[
g(x,y)=x^2-y^3-y^2.
\]
Both curves pass through the point $(0,-1)$. Intuitively, they have both the same tangent line and curvature at the point $(0,-1)$ which can be confirmed also by a computation. 

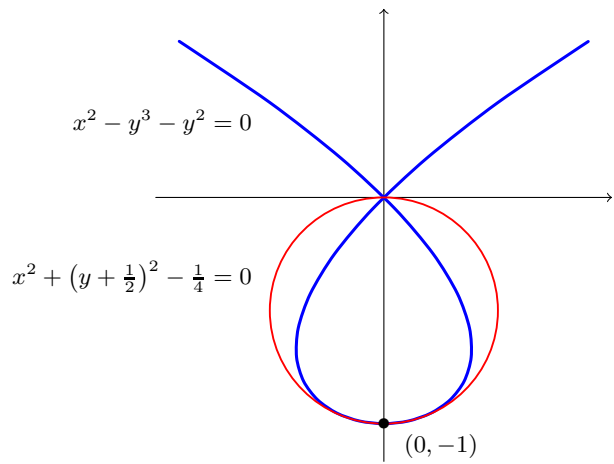
\begin{figure}[H]
\centering
\begin{tikzpicture}[line cap=round,line join=round,x=1.0cm,y=1.0cm]
\draw[color=black,-{>[scale=3,
  ]}] (-3.,0.) -- (3.,0.);

\draw[color=black,-{>[scale=3,
  ]}] (0.,-3.5) -- (0.,2.5);

\draw[scale=3,line width=0.4mm,domain=-1.3:1.3,smooth,variable=\x,blue] plot ({\x*(\x*\x-1)},{\x*\x-1});
\node[above left] at (-1.6,0.7) {\footnotesize $x^2-y^3-y^2=0$};

\draw[->,line width=0.25mm,red] (0,-3/2) circle [radius=3/2];
\node[below left] at (-1.6,-0.7) {\footnotesize $x^2+\left(y+\frac{1}{2}\right)^2-\frac{1}{4}=0$};

\draw (0,-3.3) node[anchor=south] {\Large\textbullet};
\node[below right] at (0.13,-3.0) {\footnotesize $(0,-1)$};

\end{tikzpicture}

\caption{Two different curves with the same tangent line and curvature at the point $(0,-1)$:\newline \hspace*{1.5cm} a circle (red) and a node (blue).}
\end{figure}

These two curves are, however, even at the point $(0,-1)$ \emph{very} different. In fact, if we zoom at the point $(0,-1)$ in, we observe a clear deviation of the individual arcs from each other:\\

\begin{figure}[H]
\centering
\begin{tikzpicture}[line cap=round,line join=round,x=1.0cm,y=1.0cm]

\draw[color=black,-{>[scale=3,
  ]}] (0.,-0.6) -- (0.,2.5);

\draw[scale=10,line width=0.4mm,domain=-0.5:0.5,smooth,variable=\x,blue] plot ({\x*(\x*\x-1)},{\x*\x});
\node[above left] at (-3.4,1.2) {\footnotesize $x^2-y^3-y^2=0$};

\draw[scale=7,line width=0.4mm,domain=-0.12:0.12,smooth,variable=\x,red] plot ({5*\x*(\x*\x*\x*\x*\x*\x*\x*\x*\x*\x*\x*\x*\x*\x*\x-1)},{5*\x*\x-0.004});
\node[below left] at (-3.4,0.4) {\footnotesize $x^2+\left(y+\frac{1}{2}\right)^2-\frac{1}{4}=0$};

\draw (0,-0.348) node[anchor=south] {\LARGE\textbullet};
\node[below right] at (0.17,-0.17) {\footnotesize $(0,-1)$};

\end{tikzpicture}

\caption{The deviation of the analytic branches of a circle (red) and of a node (blue) at the\newline \hspace*{1.5cm} point $(0,-1)$.}
\end{figure}
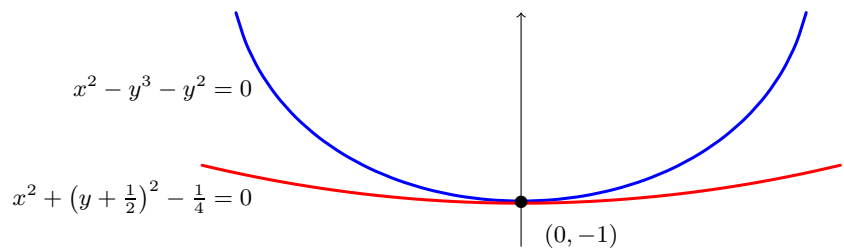

Let us formulate this observation more precisely: The analytic branches at $(0,-1)$ of the individual curves can be parametrized by
\[
\gamma_f\colon t\mapsto \left(\frac{1}{2}\cos(t),-\frac{1}{2}\sin(t)-\frac{1}{2}\right)\]
and
\[
\gamma_g\colon t\mapsto (t^3-t,t^2-1),
\]
respectively. To see that these two parametrizations indeed define different arcs, let us compose $\gamma_f$ with the analytic map $t\mapsto \arccos\bigl(2(t^3-t)\bigr)$ in order to obtain the following parametrization of the branch at $(0,-1)$ of the curve $\{f=0\}$:
\[
\gamma_{f^*}\colon t\mapsto \left(t^3-t,-\frac{\sqrt{-4(t^3 - t)^2 + 1}}{2} - \frac{1}{2}\right).
\]
Now, it is easy to compare the arcs $\gamma_{f^*}(t)$ and $\gamma_g(t)$ and to see that we have the equality
\[
\gamma_{f^*}(t)=\gamma_g(t)
\]
only for $t=0,\pm 1$, which gives the points $(0,0)$ and $(0,-1)$ on the curves. Hence, in a neighborhood of $(0,-1)$, the curves $\{f=0\}$ and $\{g=0\}$ have no points in common, except for the point $(0,-1)$ itself, and thus have different analytic branches at $(0,-1)$.

In general, one cannot reconstruct parametrizations or implicit equations of an analytic branch of a curve at a point only from the tangent vector, curvature and torsion at that point.

The goal of this note is to introduce a more refined family of geometric algebraic quantities of algebraic curves in $\mathbb{A}_{\mathbb{C}}^n$, for $n\geq 2,$ which, on one hand, can be computed from local parametrizations of curves at a given point and whose values do not depend on a choice of prametrization --- therefore geometric quantities --- and which, on the other hand, equally admit a description in terms of the defining implicit equations of curves --- therefore algebraic quantities. And finally, their values at a single smooth point of a curve provide already complete information about the curve at that point --- about its local parametrizations and also about the implicit equation defining the analytic branch of the curve.

Let us now describe briefly the constructions and results for plane curves over $\mathbb{R}$. Let us fix a plane algebraic curve $Y\subseteq\mathbb{A}^2_{\mathbb{R}}$ with $0\in Y$ a smooth point. Let $f\in\mathbb{R}[x,y]$ be the defining equation of $Y$. The slope of the tangent vector and the curvature of $Y$ at $0$ are given by the evaluation at $0$ of
\[
s(f)=-\frac{f_x}{f_y} \quad \text{ and }\quad \kappa(f)=\frac{|f_{xx}f_y^2-2f_{xy}f_xf_y+f_{yy}f_x^2|}{\sqrt{(f_x^2+f_y^2)^3}},
\]
respectively. Here $f_x, f_y, f_{xx},\dots$ denote the partial derivatives $\partial_xf,\partial_yf,\partial_{x}^2f,\dots$. Notice that as $Y$ is smooth at $0$, we can w.l.o.g.~assume that $f_{y}(0)\neq 0$ and so both, $s(f)$ and $\kappa(f)$, are well defined at $0$. They are, however, in general not sufficient to distinguish between two analytic branches of the same curve or of two different curves, so we search for further geometric algebraic quantities of the curve.

At this point it is instructive to look at $Y$ from the perspective of parametrization. So let $(x(t),y(t))$ be a formal parametrization of $Y$ at $0$. We consider $x(t),y(t)\in\mathbb{R}\ps{t}$ as formal power series and may then express $s(f)$ and $\kappa(f)$ in terms of the parametrization by
\[
s(t)=\frac{y'(t)}{x'(t)} \quad\text{ and }\quad \kappa(t)=\frac{|y''(t)x'(t)-y'(t)x''(t)|}{\sqrt{(x'(t)^2+y'(t)^2)^3}}.
\]
Their evaluation at $t=0$ gives again the slope of the tangent vector and the curvature of $Y$ at $0$, respectively. Notice that due to the fact that the pair $(x(t),y(t))$ parametrizes $Y$ at a smooth point and the assumption $f_y(0)\neq 0$, we have $\ord(x(t))=1$, or equivalently $x'(0)\neq 0$. To work with algebraic objects, we modify the second formula in order to avoid roots, getting the \emph{algebraic curvature}:
\[
\kappa_1(t)\coloneqq\frac{y''(t)x'(t)-y'(t)x''(t)}{x'(t)^3}.
\]
(We could take as the denominator any non-trivial linear combination of $x'(t)$ and $y'(t)$ but for further computations it is more convenient to take $x'(t)$.) Now comes the crucial idea: By definition, the tangent vector and also the classical curvature of $Y$ at $0$ are geometrically defined objects associated to $Y$ at $0$ and as such, their values at $0$ do not depend on the choice of parametrization. This is due to the fact that the curvature is (classically) defined as the inverse of the radius of the osculating circle for $Y$ at $0$.

The formulas for the slope of the tangent vector and the curvature in terms of parametrization therefore do not depend on the chosen parametrization and this also holds for the algebraic curvature: the slope of the tangent vector and the algebraic curvature are rational expressions in $x(t)$ and $y(t)$ and their derivatives that are equivariant under the natural action of the group of reparametrizations $\Aut(\mathbb{R}\ps{t})$. We call such an equivariant expression a \emph{geometric invariant}. More precisely, for any reparametrization $\varphi$, the slope of the tangent vector satisfies the equality
\[
s(\varphi(t))=\frac{y'(\varphi(t))}{x'(\varphi(t))}=\frac{(y\circ\varphi)'(t)}{(x\circ\varphi)'(t)}.
\]
This shows that the slope of the tangent vector at a point does not change with reparametrizations of the curve. The same holds also for the algebraic curvature. From this it follows intuitively (and can be proven rigorously, see Theorem \ref{impl}) that $\kappa_1(t)$ must have also an expression as a formula in terms of $f(x,y)$ and its partial derivatives. Actually, the implicit formula for $\kappa_1(t)$ is given by
\[
\kappa_1(f)\coloneqq-\frac{f_{xx}f_y^2-2f_{xy}f_xf_y+f_{yy}f_x^2}{f_y^3}.
\]
This already shows that the algebraic curvature defines an algebraic quantity of the curve. The key observation now is that the parametric formula for the algebraic curvature can be computed from the slope of the tangent vector in the following way:
\[
\kappa_1(t)=\frac{\partial_t(s(t))}{x'(t)}.
\]
Hence, by differentiating the algebraic curvature w.r.t.~$t$ and dividing by $x'(t)$, in order to ensure again equivariance under reparametrizations, and repeating this process, we iteratively define the following family of parametric expressions
\[
\kappa_i(t)\coloneqq\frac{\partial_t(\kappa_{i-1}(t))}{x'(t)}.
\]
By construction, they are all equivariant under reparametrizations and thus admit also an implicit expression as rational functions in terms of $f(x,y)$ and its partial derivatives (as proven in Theorem \ref{impl}), i.e., there exist rational functions $\widetilde{\kappa}_i$ in the partial derivatives of $f$, i.e.,
\[
\widetilde{\kappa}_i\in\mathbb{R}(f_x,f_y,f_{xx},\dots)\subseteq\mathbb{R}(x,y),
\]
whose evaluation at the pair $(x(t),y(t))$ equals $\kappa_i(t)$:
\begin{align*}\label{impl:ex}
\widetilde{\kappa}_i(x(t),y(t))=\kappa_i(t).
\end{align*}
These constructions extend to complex space curves in $\mathbb{A}_\mathbb{C}^{n+1}$ in the following way: Let us consider an algebraic space curve $X\subseteq\mathbb{A}^{n+1}_\mathbb{C}$ with a smooth point $0\in X$ and let 
\[
\gamma\colon t\mapsto(x(t),y_1(t),\dots,y_n(t))
\]
be its analytic parametrization at $0$. For each $j=1,\dots,n$, we define
\[
s_j\coloneqq\frac{y_j'(t)}{x'(t)},
\]
and call them the \emph{slopes} of $X$, and 
\[
\kappa_{1,j}(t)\coloneqq\frac{y_j''(t)x'(t)-y_j'(t)x''(t)}{x'(t)^3} \quad\text{ and }\quad \kappa_{i,j}(t)\coloneqq\frac{\partial_t\kappa_{i-1,j}(t)}{x'(t)},
\]
the \emph{first} and \emph{higher algebraic curvatures} of $X$, respectively. We call again each rational expression in $x(t),y_1(t),\dots,y_n(t)$ and their higher derivatives, that is equivariant under reparametrizations, a \emph{geometric invariant}.

We will prove in the body of this paper that each geometric invariant can be written as a rational function in $x(t),y_j(t),s_j(t)$ and $\kappa_{i,j}(t),i\geq 1$, for $j=1,\dots,n$  and that it admits an implicit description. More precisely: 

\begin{theorem}\label{main1}
Let $\gamma(t)=(x(t),y_1(t),\dots,y_n(t))\in\mathbb{C}\ps{t}^{n+1}$ be a parametrization of a (not necessarily smooth) space algebraic curve $X\subseteq\mathbb{A}_\mathbb{C}^{n+1}$. Further, let $f_1,\dots,f_r\in\mathbb{C}[x,y_1,\dots,y_n]$ be its defining polynomials. Then for $p(t)$ a rational function in the higher derivatives of the components of $\gamma(t)$, the following statements are equivalent:
\begin{itemize}
    \item[(i)] $p(t)$ is a geometric invariant, i.e., it is equivariant under reparametrizations.
    \item[(ii)] $p(t)$ can be written as a rational function in $s_j(t)$ and $\kappa_{i,j}(t),$ for $1\leq j\leq n$ and $i\geq 1$. Moreover, $p(t)$ is invariant under the substitution
    \begin{align*}
i_\kappa\colon &x'(t)\mapsto 1, \partial^i_tx(t)\mapsto 0\ \text{ for all }\ i\geq 2\\
&y'_j(t)\mapsto s_j(t),\partial_t^iy_j(t)\mapsto \kappa_{i-1,j}(t)\ \text{ for all }\ i\geq 2, \text{ and } j=1,\dots, n.
\end{align*}
    \item[(iii)] There exist polynomials 
\[
\tilde{p}_1,\tilde{p}_2\in\mathbb{C}[\partial_{x}^{i_0}\partial_{y_1}^{i_1}\cdots\partial_{y_n}^{i_n}f_k\,:\,i_0,\dots,i_n\in\mathbb{N},1\leq k\leq r]\subseteq\mathbb{C}[x,y_1,\dots,y_n],
\]
such that the equality
\[
p(t)=\frac{\tilde{p}_1(\gamma(t))}{\tilde{p}_2(\gamma(t))}
\]
is fulfilled. In other words, $p(t)$ admits an implicit expression as a rational function in terms of the higher partial derivatives of the defining equations $f_1,\dots,f_r$ of $X$.
\end{itemize}
\end{theorem}

The equivalence between $(i)$ and $(ii)$ is discussed in Theorems \ref{12} and \ref{14}, and the equivalence between $(i)$ and $(iii)$ then in Theorems \ref{impl} and \ref{impl:sp}.

Further, using Theorem \ref{main1}, we explain how to reconstructs a local parametrization of $X$ at the origin from the values at $0$ of the implicit expressions $\widetilde{s}_j$ and $\widetilde{\kappa}_{i,j}, i\geq 1$, of $s_j(t)$, and $\kappa_{i,j}(t), i\geq 1$, respectively. As all $\widetilde{s}_j$ and $\widetilde{\kappa}_{i,j}, i\geq 1$ are rational functions in the partial derivatives of the defining polynomials $f_1,\dots,f_r$ of the curve $X$, this gives a description of the local analytic equations of $X$ at $0$ in terms of the defining polynomials $f_1,\dots,f_r$:

\begin{corollary}\label{main2}
Let $X\subseteq\mathbb{A}_\mathbb{C}^{n+1}$ be an algebraic curve with $0\in X$ a smooth point. Further let $\gamma(t)=(x(t),y_1(t),\dots,y_n(t))\in\mathbb{A}_\mathbb{C}^{n+1}$ be a parametrization of $X$ at $0$. Then the analytic branch of $X$ at $0$ is defined by the ideal
\[
\mathcal{I}_X=\left(y_j-\widetilde{s}_j(0)\cdot x-\sum_{i\geq 1}\frac{\widetilde{\kappa}_{i,j}(0)}{(i+1)!}\cdot x^{i+1}  \,:\, 1\leq j\leq n\right),
\]
where $\widetilde{s}_j$ and $\widetilde{\kappa}_{i,j}$ are implicit expressions of $s_j(t)$ and $\kappa_{i,j}(t)$, respectively, as rational functions in the higher partial derivatives of the defining polynomials $f_1,\dots,f_r$ of $X$,  i.e.,
\[
\widetilde{s}_j,\widetilde{\kappa}_{i,j}\in\mathbb{C}[\partial_{x}^{i_0}\partial_{y_1}^{i_1}\cdots\partial_{y_n}^{i_n}f_k\,:\,i_0,\dots,i_n\in\mathbb{N},1\leq k\leq r]\subseteq\mathbb{C}[x,y_1,\dots,y_n]
\]
satisfy the equalities
\[
s_j(t)=\widetilde{s}_j(\gamma(t))\quad\text{ and }\quad\kappa_{i,j}(t)=\widetilde{\kappa}_{i,j}(\gamma(t)),
\]
respectively.

\end{corollary}

For the proof see Corollaries \ref{reconstr} and \ref{sp:cur}. Let us mention at this point that the description of analytic branches of $X$ provided by this corollary is very helpful in resolution of singularities as it shows that the algebraic curvatures are looking very closely into the local geometry of the curve. Moreover, it can be shown (the precise statement is presented in \cite{me}) that the information provided by them is crucial for defining a suitable center of the blowup which improves the singularities of the curve in a much faster way than the standard blowup or the Nash modification.

The concept of algebraic curvatures and geometric invariants as well as the results mentioned above are defined and stated in the next sections of this note in a more general setting which allows applications in a broader context and make them to an even more powerful mathematical tool. This approach allows applications also outside algebraic geometry --- the geometric invariants can be for instance used to prove the First Fundamental Theorem for $\GL_n(\mathbb{C})$ as presented in \cite{mel}. We introduce in this note the concept of geometric invariants using only the language of differential fields, symbolic derivatives and symbolic chain rule. 

We begin with introducing the concept of geometric invariants of plane curves and with a detailed investigation of their basic properties. The techniques used in the case of plane curves apply then also to the space curve case and allow us a very natural extension of the results gained about geometric invariants of plane curves to geometric invariants of space curves. 

For the study of geometric invariants and their basic properties we will use the existence of differentially transcendental power series and their families. We collect some important facts about them and also proofs of their existence in Appendix \ref{appendix}.

\section{Geometric invariants of Plane Curves}\label{section2}

The aim of this section is to introduce the geometric invariants of plane curves and to study their basic properties. As already mentioned in the introduction, in this section, we will extend the concept of under reparametrizations equivariant rational functions in the components of a parametrization $\gamma(t)=(x(t),y(t))\in\mathbb{C}\ps{t}^2$ of a plane algebraic curve and their derivatives to a more general setting. We replace the derivatives $\partial_t^ix(t)$ and $\partial_t^iy(t)$ for $i\in\mathbb{N}$ by new variables $x^{(i)}$ and $y^{(i)}$, respectively, and translate the property of being equivariant under reparametrizations into a new invariance property on these variables. We study the structure of the field of these invariants and provide with Theorem \ref{12} a minimal countable system of generators over $\mathbb{C}$. Further, we show in Theorem \ref{impl} that these invariants can be described in terms of the implicit equation of the plane curve and that they already uniquely determine its smooth analytic branches, see Corollary \ref{reconstr}.

Let us consider two sets of countably many variables $x^{(i)}$ and $y^{(i)},i\in\mathbb{N}$ (we think of $x^{(i)}$ as a symbolic derivative of $x^{(i-1)}$). We consider the field 
\[
F\coloneqq \mathbb{C}(x^{(i)},y^{(i)}  \,:\, i\in\mathbb{N})
\]
generated by all $x^{(i)},y^{(i)}$, equipped with the $\mathbb{C}$-derivation 
\begin{align*}
\partial \colon F&\rightarrow F \\
x^{(i)}&\mapsto x^{(i+1)}\\
y^{(i)}&\mapsto y^{(i+1)}.
\end{align*} 
It thus becomes a differential field $(F,\partial)$. Let $\varphi^{(i)},i \in\mathbb{N}$, be another set of variables (they play a different role than $x^{(i)},y^{(i)}$). Let 
\[
L\coloneqq F(\varphi^{(i)}  \,:\, i\in\mathbb{N})=\mathbb{C}(x^{(i)},y^{(i)},\varphi^{(i)}  \,:\, i\in\mathbb{N}),
\]
and extend  $\partial$ to $L$ by $\partial(\varphi^{(i)})=\varphi^{(i+1)}.$  On $L$ we simulate the chain rule by another $\mathbb{C}$-derivation:
\begin{align*}
\chi\colon L&\rightarrow L \\
x^{(i)}&\mapsto x^{(i+1)}\varphi^{(1)}\\
y^{(i)}&\mapsto y^{(i+1)}\varphi^{(1)} \\
\varphi^{(i)}&\mapsto \varphi^{(i+1)}.
\end{align*} 
We think of $\varphi$ as a symbol for reparametrization of a parametrized curve. More precisely, given a parametrized curve $\gamma(t)=(x(t),y(t))\in\mathbb{C}\ps{t}^2$, we associate
\begin{align*}
x^{(0)}&\leftrightarrow x(t),\\
y^{(0)}&\leftrightarrow y(t),
\end{align*} 
and
\begin{align*} 
x^{(i)}&\leftrightarrow \partial_t^ix(t),\\
y^{(i)}&\leftrightarrow \partial_t^iy(t),
\end{align*} 
for $i\geq 1$. Let $\varphi\in\Aut(\mathbb{C}\ps{t})$ be an algebra automorphism. We call each such $\varphi$ a \emph{reparametrization}. Note that $\varphi$ is given by a power series $\varphi(t)\in\mathbb{C}\ps{t}$ with $\ord(\varphi(t))=1$. The automorphism group $\Aut(\mathbb{C}\ps{t})$ acts then from the right on $\mathbb{C}\ps{t}^2$ via
\begin{align*}
\Aut(\mathbb{C}\ps{t})\times\mathbb{C}\ps{t}^2&\rightarrow\mathbb{C}\ps{t}^2\\
\bigl(\varphi,(x(t),y(t))\bigr)&\mapsto\bigl(x(\varphi(t)),y(\varphi(t))\bigr).
\end{align*}
From now on, by a reparametrization $\varphi$ we always mean the power series representation $\varphi(t)$ of the automorphism $\varphi\in\Aut(\mathbb{C}\ps{t})$. We associate 
\[
\varphi^{(0)}\leftrightarrow \varphi(t)
\] 
and 
\[
\varphi^{(i)}\leftrightarrow\partial_t^i\varphi(t)
\]
for $i\geq 1$. With this, the derivation $\chi$ reflects the chain rule 
\[
\partial_t(\partial_t^ix\circ\varphi)(t)=(\partial^{i+1}_tx\circ\varphi)(t)\cdot\varphi'(t).
\]
Next we define a $\mathbb{C}$-morphism, i.e., a field homomorphism whose fixed field equals $\mathbb{C}$, on $L$ by 
\begin{align*}
\Lambda\colon L&\rightarrow L \\
x^{(i)}&\mapsto\chi^i(x^{(0)})\\
y^{(i)}&\mapsto\chi^i(y^{(0)})\\
\varphi^{(i)}&\mapsto\chi^i(\varphi^{(0)}),
\end{align*}
where $\chi^i$ denotes the composition $\underbrace{\chi\circ\dots\circ\chi}_{i-\text{times}}$. In terms of power series, this means 
\begin{align*}
\partial_t^ix(t)&\mapsto\partial^i_t(x\circ\varphi)(t),\\
\partial_t^iy(t)&\mapsto\partial^i_t(y\circ\varphi)(t),\\
\partial_t^i\varphi(t)&\mapsto\partial^i_t\varphi(t).
\end{align*}
So for $x(t)$ and $y(t)$, their higher derivatives are replaced by the derivatives of the compositions $(x\circ\varphi)(t)$ and $(y\circ\varphi)(t)$, respectively, for which the iterated chain rule applies. From now on we will denote the vectors $(x^{(0)},x^{(1)},\dots)$ and $(y^{(0)},y^{(1)},\dots)$ by $\underline{x}$ and $\underline{y}$, respectively, and for $p(x^{(i)},y^{(i)}:i\in\mathbb{N})\in F$ we will use the notation $p(\underline{x},\underline{y})$.

\begin{definition}\normalfont
A rational expression $p\in F$ is called a \emph{geometric invariant of plane curves} if it is fixed under $\Lambda,$ namely 
\[
\Lambda(p)=p.
\]
\end{definition}

In terms of power series (parametrizations) this means
\begin{align}\label{1}
p\bigl(\partial^i_tx(t),\partial_t^iy(t): i\in\mathbb{N}\bigr)\circ\varphi(t)=p\bigl(\partial^i_t(x\circ\varphi)(t),\partial^i_t(y\circ\varphi)(t): i\in\mathbb{N}\bigr).
\end{align}
Being a geometric invariant is hence a property which reflects the equivariance of a rational expression in power series $x(t),y(t)$ and their derivatives under reparametrizations. We denote the vectors $(x(t),\partial_tx(t),\partial^2_tx(t),\dots)$ and $(y(t),\partial_ty(t),\partial^2_ty(t),\dots)$ by $\underline{x(t)}$ and $\underline{y(t)}$, respectively.

\begin{proposition}\label{param}
Let $p(\underline{x},\underline{y})=\frac{g(\underline{x},\underline{y})}{h(\underline{x},\underline{y})}$ be an element of $F=\mathbb{C}(x^{(i)},y^{(i)}  \,:\, i\in\mathbb{N})$. Then the following statements are equivalent:
\begin{itemize}
    \item[(i)] $p$ is a geometric invariant of plane curves.
    \item[(ii)] The equality
     \[
    p\bigl(\underline{x(t)},\underline{y(t)}\bigr)\circ\varphi(t)=p\bigl(\underline{(x\circ\varphi)(t)},\underline{(y\circ\varphi)(t)}\bigr)
    \]
    holds for all power series $x(t),y(t)\in\mathbb{C}\ps{t}$ with $h\bigl(\underline{x(t)},\underline{y(t)}\bigr)\neq 0$ and all reparametrizations $\varphi(t)$, i.e., $p\bigl(\underline{x(t)},\underline{y(t)}\bigr)$ is equivariant under reparametrizations.
\end{itemize}
\end{proposition}

The proof of this proposition was obtained in collaboration with A.~Bostan.

\begin{proof}
$(ii)\Rightarrow(i)$: Let $x(t),y(t),\varphi(t)\in\mathbb{C}\ps{t}$ be a family of D-algebraically independent power series (see Appendix \ref{appendix} for the definition and basic properties of D-algebraically independent power series) satisfying the condition $\partial_t\varphi(0)\neq 0$, i.e., $\varphi(t)$ defines a reparametrization. Then the higher derivatives $\partial_t^ix,\partial_t^iy$ and $\partial_t\varphi$ do not satisfy any polynomial equation and clearly this remains true also after a reparametrization. Thus, they can be considered as variables $x^{(i)},y^{(i)},\varphi^{(i)}$. Then according to the chain rule and after rewriting the derivatives $\partial_t^ix\circ\varphi, \partial_t^iy\circ\varphi, \partial_t^i\varphi$ as $x^{(i)},y^{(i)},\varphi^{(i)}$, respectively, we have
\begin{align*}
p(\underline{x},\underline{y})&=p\bigl(\underline{x(t)},\underline{y(t)}\bigr)\circ\varphi=p\bigl(\underline{(x\circ\varphi)(t)},\underline{(y\circ\varphi)(t)}\bigr)\\
&=\Lambda(p)\bigl(\underline{x(t)}\circ\varphi,\underline{y(t)}\circ\varphi,\underline{\varphi(t)}\bigr)=\Lambda(p)(\underline{x},\underline{y},\underline{\varphi}).
\end{align*}
From this we conclude $p=\Lambda(p)$, which shows that $p$ is a geometric invariant of plane curves.
$(i)\Rightarrow (ii)$: This follows from the fact that reparametrizations act exactly in the same way as~$\Lambda$.
\end{proof}

\begin{lemma}\label{11}
For an element $p\in F$ we have the following two equalities
\begin{itemize} 
\item[(i)] $\chi(p)=\partial(p)\varphi^{(1)}$,
\item[(ii)] $\Lambda(\partial(p))=\chi(\Lambda(p)).$
\end{itemize} 
\end{lemma}

\begin{proof}
$(i)\colon$ As $\chi$ is a derivation and acts on the generators of $F$ by $\chi(x^{(i)})=\partial(x^{(i)})\varphi^{(1)}$ and $\chi(y^{(i)})=\partial(y^{(i)})\varphi^{(1)}$, the claimed equality follows.\\
$(ii)\colon$ Take an element $p\in F$. Since the maps $\Lambda, \partial$ and $\chi$ are additive, we may assume that $p$ is of the form 
\begin{align*}
p=\prod_{i\in I}x^{(i)}\prod_{j\in J}y^{(j)}
\end{align*}
with some index sets $I$ and $J$. A short computation shows then 
\begin{align*}
\Lambda(\partial(p))=&\Lambda\left(\sum_{i\in I}\bigl(x^{(i+1)}\prod_{k\in I\backslash\{i\}}x^{(k)}\prod_{j\in J}y^{(j)}\bigr)+\sum_{j\in J}\bigl(y^{(j+1)}\prod_{i\in I}x^{(k)}\prod_{k\in J\backslash\{j\}}y^{(k)}\bigr)\right)\\
=&\sum_{i\in I}\bigl(\chi^{(i+1)}(x^{(0)})\prod_{k\in I\backslash\{i\}}\chi^{(k)}(x^{(0)})\prod_{j\in J}\chi^{(j)}(y^{(0)})\bigr)+\\
&+\sum_{j\in J}\bigl(\chi^{(j+1)}(y^{(0)})\prod_{i\in I}\chi^{(k)}(x^{(0)})\prod_{k\in J\backslash\{j\}}\chi^{(k)}(y^{(0)})\bigr)\\
=&\chi\bigl(\prod_{i\in I}\chi^{(i)}(x^{(0)})\prod_{j\in J}\chi^{(j)}(y^{(0)})\bigr)=\chi\bigl(\Lambda(\prod_{i\in I}x^{(i)}\prod_{j\in J}y^{(j)})\bigr)=\chi(\Lambda(p)).\qedhere
\end{align*}
\end{proof}

Using the last proposition and lemma, we can construct a whole family of geometric invariants:

\begin{example}\label{ex}\normalfont The variables $x^{(0)},y^{(0)}$ are of course geometric invariants. But there are more interesting examples.
\begin{itemize}
\item[(1)] The slope of the tangent vector \[
s(t)=\frac{y'(t)}{x'(t)}
\]
of the parametrized curve $\gamma(t)$ is obviously equivariant under reparametrizations.
Hence, 
\[
\kappa_0\coloneqq \frac{y^{(1)}}{x^{(1)}}
\]
is a geometric invariant, called the \emph{slope} (\emph{of the tangent vector}).
\item[(2)] The formula for the classical curvature 
\[
\kappa(t)=\frac{|y''(t)x'(t)-y'(t)x''(t)|}{\sqrt{\bigl(x'(t)^2+y'(t)^2\bigr)^3}}
\]
does not yield a geometric invariant in our sense (although it is equivariant under reparametrizations) since we do not allow square roots in our definition. However, a little modification leads to the rational expression 
\[
\frac{y''(t)x'(t)-y'(t)x''(t)}{\bigl(x'(t)+y'(t)\bigr)^3}
\]
which is also equivariant under reparametrizations. Note that any linear combination $ax'(t)+by'(t)$ with $a\neq  0$ or $b\neq 0$ in the denominator also yields a geometric invariant. (For further computations it is convenient to choose $a=1,b=0.$) We set 
\[
\kappa_1\coloneqq\frac{y^{(2)}x^{(1)}-y^{(1)}x^{(2)}}{(x^{(1)})^3}.
\]
\item[(3)] The rational expressions 
\[
\kappa_{i+1}\coloneqq\frac{\partial(\kappa_i)}{x^{(1)}} \ \text{ for any } i\geq 1,
\]
are again geometric invariants. This follows directly from Lemma \ref{11}. 
\end{itemize}
\end{example}

The slope of the tangent vector $s(t)$ and the curvature $\kappa(t)$ of parametric plane and also space curves over $\mathbb{R}$ are well-known and standard notions in differential geometry. The classical literature (for instance, R.~Goldman in \cite{go05}, or M.~P.~do Carmo in \cite[Chapter 1, \S5]{ca76}) often refers to them as ``invariants (under reparametrizations)''. So the classical curvature $\kappa$ is the differential geometric analogue to the more algebraically defined geometric invariant $\kappa_1$. 

Since the geometric invariant $\kappa_1$ constructed above was derived from the formula for the classical curvature, we call $\kappa_1$ the \emph{first algebraic curvature} (\emph{of plane curves}) and $\kappa_i,i>1$ the \emph{higher algebraic curvatures} (\emph{of plane curves}). Notice that since $\Lambda$ is a field homomorphism, the geometric invariants of plane curves form a field. We set 
\[
I_F\coloneqq \text{field of geometric invariants (of plane curves)}.
\]
The algebraic curvatures do not only represent a family of geometric invariants, but even more, they generate the whole field $I_F$.

\begin{theorem}\label{12}
The field of geometric invariants of plane curves is generated over $\mathbb{C}$ by the variables $x^{(0)},y^{(0)}$, the slope and the first and higher algebraic curvatures, i.e., we have 
\[
I_F=\mathbb{C}(x^{(0)},y^{(0)}, \kappa_i \,:\, i\in\mathbb{N}).
\]
\end{theorem}

We prove the theorem in two different manners. The first proof presented here uses differential field extensions. This proof was done in collaboration with J.~Schicho and M.~Gallet. The second proof was obtained in cooperation with A.~Bostan and follows a different strategy. It uses Proposition \ref{param} --- the correspondence between geometric invariants and rational expressions in parametrizations and their higher derivatives which are equivariant under reparametrizations --- and the existence of D-algebraically independent families of power series. The second proof will show us even more. Given a geometric invariant $p$, the second proof tells us how to find the representation of $p$ as a rational function in the generators $x^{(0)},y^{(0)}$ and $\kappa_i,i\in\mathbb{N}.$ 

Let us now start with the first proof. Set 
\[J\coloneqq\mathbb{C}(x^{(0)},y^{(0)},\kappa_i  \,:\, i\in\mathbb{N})\subseteq F\quad \text{ and }\quad \mathsf{X}\coloneqq\{x^{(i)}  \,:\, i\in\mathbb{N}, i\neq 0\}\subseteq F.
\]

\begin{lemma}\label{15}
The field $F$ is generated by $\mathsf{X}$ over the subfield $J$, i.e., $F=J(\mathsf{X}).$
\end{lemma}

\begin{proof}
From $\partial(\kappa_i)=\kappa_{i+1} x^{(1)}\in J(\mathsf{X})$ and $\partial(\mathsf{X})\subseteq \mathsf{X}$, it follows that $J(\mathsf{X})$ is closed under $\partial$, i.e., $\partial(J(\mathsf{X}))\subseteq J(\mathsf{X})$. Further, since $x^{(0)},y^{(0)}\in J(\mathsf{X})$, all higher derivatives $x^{(i)},y^{(i)},i\geq 1$ are elements of $J(\mathsf{X})$ as well. Hence, the generators of $F$ lie in $J(\mathsf{X})$ and so $F\subseteq J(\mathsf{X})$. Since $J(\mathsf{X})\subseteq F$ by construction, the statement is proven.
\end{proof}

For each $i\geq 1$ we set $V_i\coloneqq F[\varphi^{(j)}  \,:\, 1\leq j\leq i]\subseteq L$. Notice that for each $i$ we have the inclusion $\chi(V_i)\subseteq V_{i+1}.$

\begin{lemma}\label{16}
For each $i\geq 1$ we have $\Lambda(x^{(i)})-\varphi^{(i)}x^{(1)} \in V_{i-1}$.
\end{lemma}

\begin{proof}
We proceed by induction. The induction base follows immediately from 
\[
\Lambda(x^{(1)})=\varphi^{(1)}x^{(1)}.
\]
For $i\geq 2$ we use Lemma \ref{11} and obtain 
\begin{align*}
\Lambda(x^{(i+1)})-\varphi^{(i+1)}x^{(1)}&=\Lambda\bigl(\partial(x^{(i)})\bigr)-\varphi^{(i+1)}x^{(1)}=\chi\bigl(\Lambda(x^{(i)})\bigr)-\varphi^{(i+1)}x^{(1)}.
\end{align*}
According to the induction hypothesis, we have $\Lambda(x^{(i)})=\varphi^{(i)}x^{(1)}+v_{i-1}$ for some element $v_{i-1}\in V_{i-1}$. Thus, we can express
\begin{align*}
\chi\bigl(\Lambda(x^{(i)})\bigr)-\varphi^{(i+1)}x^{(1)}=\chi(\varphi^{(i)}x^{(1)}+ v_{i-1})-\varphi^{(i+1)}x^{(1)}=\varphi^{(i)}\varphi^{(1)}x^{(2)}+ \chi(v_{i-1}).
\end{align*}
But $\varphi^{(i)}\varphi^{(1)}x^{(2)}+ \chi(v_{i-1})\in V_i$ and the claim follows.
\end{proof}

\begin{proposition}\label{17}
The set $\mathsf{X}$ is algebraically independent over $I_F$.
\end{proposition}

\begin{proof}
Let us assume indirectly that $\mathsf{X}$ is algebraically dependent over $I_F$. Let $m\in \mathbb{N}$ be the minimal positive integer such that there exists a polynomial in $m$ variables 
\[
g(w)\in I_F[w]=I_F[w_1,\dots,w_m],g(w)\neq 0 \ \text{ with }\ g(\mathsf{X})\coloneqq g(x^{(1)},\dots,x^{(m)})=0.
\]
Let us denote by $w'$ the vector $(w_1,\dots,w_{m-1})$ and write 
\[
g(w)=\sum_j \alpha_j(w')w_m^j
\]
as a polynomial in $w_m$ with coefficients in $I_F[w']=I_F[w_1,\dots,w_{m-1}]$. Set 
\[
\beta_j:=\Lambda\bigl(\alpha_j(\mathsf{X})\bigr)=\Lambda\bigl(\alpha_j(x^{(1)},\dots,x^{(m-1)})\bigr).
\]
Then we get the equality 
\[
0=\Lambda\bigl(g(\mathsf{X})\bigr)=\sum_j\beta_j\cdot\bigl(\Lambda(x^{(m)})\bigr)^j.
\]
Notice that $\beta_j\in V_{m-1}$ and with Lemma \ref{16} it holds also $\Lambda(x^{(m)})-\varphi^{(m)}x^{(1)}\in V_{m-1}$. All this applied to the above equality gives us
\[
V_{m-1}\ni \sum_j\beta_j\cdot\bigl((\Lambda(x^{(m)}))^j-(\varphi^{(m)})^j(x^{(1)})^j\bigr)=0-\sum_j\beta_j\cdot(\varphi^{(m)})^j(x^{(1)})^j,
\]
whence follows $\beta_j=0$ for all $j\geq 1$ because no power $\varphi^{(m)}$ belongs to $V_{m-1}$.  Since $\Lambda$ is a field homomorphism, it is injective, so $\alpha_j(\mathsf{X})=0$ for all $j\geq 1$. But as $m$ was chosen minimal, we have $\alpha_j(w')=0$ for any $j\geq 1$. Further, as $g(\mathsf{X})=0$, $g$ cannot have a constant term and therefore $g(w)=0$, a contradiction.
\end{proof}

\begin{proof}[First Proof of Theorem \ref{12}]
The inclusion $J\subseteq I_F$ is clear. It thus remains to show the inclusion $I_F\subseteq J.$ Let $p\in I_F$ be a geometric invariant of plane curves. Since $I_F$ is generated as a field over $J$ by $\mathsf{X}$, there exist $f,g\in J[w]=J[w_1,\dots,w_n], g(w)\neq 0$, polynomials in $n$ variables for some $n\in\mathbb{N}$, with $p=\frac{f(\mathsf{X})}{g(\mathsf{X})}$. From $J\subseteq I_F$ and from Proposition \ref{17} we conclude that $0=f(w)-pg(w)\in I_F[w].$ Now, as $g(w)\neq 0$, the comparison of coefficients in the equality $0=f(w)-pg(w)$  yields $p\in \Quot(J)=J$ which finishes the proof.
\end{proof}

Now we move to the second and more geometric proof of Theorem \ref{12}. Consider the $\mathbb{C}$-morphism
\begin{align*}
i_{\kappa}\colon F&\rightarrow F\\
x^{(0)}&\mapsto x^{(0)}, x^{(1)}\mapsto 1,x^{(i)}\mapsto 0\ \text{  for all } i\geq 2,\\
y^{(0)}&\mapsto y^{(0)}, y^{(i)}\mapsto \kappa_{i-1}\ \text{ for all } i\geq 1.
\end{align*}
The goal is to prove that each geometric invariant stays invariant under $i_\kappa$.

\begin{proposition}\label{13}
For each geometric invariant of plane curves $p\in I_F$ we have the following equality 
\[
p=i_{\kappa}(p).
\] 
\end{proposition}

Once Proposition \ref{13} is proven, the statement of Theorem \ref{12} follows immediately. Let us mention that the key argument in the following proof is inspired by the idea used by J.-P.~Demailly in the proof of \cite[Theorem 6.8]{de97}.

\begin{proof}[Second Proof of Theorem \ref{12} and Proposition \ref{13}]
Let $p\in I_F$ be a geometric invariant of plane curves. Then for all power series $x(t),y(t)\in\mathbb{C}\ps{t}$, equality (\ref{1}) is satisfied by $p$. Let us choose $x(t)$ and $y(t)$ to be D-algebraically independent and such that $\ord(x(t))=1$. Denote by $\varphi(t)$ the compositional inverse of $x(t)$, i.e., the power series satisfying $(x\circ\varphi)(t)=t$. Applying the chain rule to the composition $x\circ\varphi$ gives 
\[
\varphi'(t)=\frac{1}{x'(\varphi(t))}
\] 
and so for the first derivative of $y\circ\varphi$ we have the equality
\[
(y\circ\varphi)'(t)=\frac{y'(\varphi(t))}{x'(\varphi(t))}=\kappa_0(\underline{x(t)},\underline{y(t)})\circ\varphi(t).
\]
For the higher derivatives, with the iterated chain rule, by induction we get 
\[
\partial^i_t(y\circ\varphi)=\kappa_{i-1}\bigl(\underline{x(t)},\underline{y(t)}\bigr)\circ\varphi(t)\ \text{ for all } i\geq 2.
\] 
Further we have
\begin{align*}
&\partial_t(x\circ\varphi)(t)=1,\\
&\partial_t^i(x\circ\varphi)(t)=0\ \text{ for all } i\geq 2.
\end{align*}
For $x(t),y(t)$ and $\varphi(t)$ as above,  equality (\ref{1}) composed with $\varphi^{-1}$ on both sides becomes then
\begin{align*}
p(\underline{x(t)},\underline{y(t)})&=p(\underline{x(t)},\underline{y(t)})\circ(\varphi\circ\varphi^{-1})(t)\\
&=p\bigl(\underline{(x\circ\varphi)(t)},\underline{(y\circ\varphi)(t)}\bigr)\circ\varphi^{-1}(t)\\
&=i_\kappa(p)(\underline{x(t)},\underline{y(t)})\circ(\varphi\circ\varphi^{-1})(t)\\
&=i_\kappa(p)(\underline{x(t)},\underline{y(t)})
\end{align*}
and thus,
\[
\bigl(p-i_\kappa(p)\bigr)(\underline{x(t)},\underline{y(t)})=0.
\]
But since $x(t),y(t)$ were chosen to be D-algebraically independent, it follows 
\[
p-i_\kappa(p)=0,
\] 
which finishes the proof.
\end{proof}

So for a given geometric invariant of plane curves $p\in I_F$, Theorem \ref{12} ensures that it can be written as a rational function in $x^{(0)},y^{(0)}$, the slope and algebraic curvatures and Proposition \ref{13} explains how to construct such a rational function. Namely one can replace each $x^{(1)}$ by $1$, $x^{(i)}$, for $i\geq 2$, by $0$ and each $y^{(i)}$, with $i\geq 1$ by $\kappa_{i-1}$ to obtain the required representation of $p$ as a rational function in the generators $x^{(0)},y^{(0)},\kappa_i$, with $i\in\mathbb{N}$, of $I_F$.\\
\vspace{3mm}

\centerline{\textbf{Implicit Expressions of Geometric Invariants of Plane Curves}}

\vspace{3mm}
As next, we will discuss the interaction between the parametric and implicit representation of plane curves and its impact on (the implicit formulas for) geometric invariants. The equation (\ref{1}) shows that for a parametrized curve $\gamma(t)=(x(t),y(t))$, each geometric invariant yields a geometric quantity which does not depend on a chosen parametrization. Hence, it should be possible to describe each such quantity given by a geometric invariant also without using local parametrizations of a plane algebraic curve, namely, just by its defining implicit equation. In the remaining part of this section, we prove that such a description in terms of the defining implicit equation is always possible, and, moreover, we provide also implicit formulas for the slope and algebraic curvatures, the generators of the field of geometric invaraints. Once their implicit expressions are known, one is able to find an implicit expression for an arbitrary geometric invariant.

Consider a plane algebraic curve $X=V(f),f\in\mathbb{C}[x,y],$ defined by a square-free polynomial $f$ with $f(0,0)=0$. We assume that $0$ is a smooth point of $X$. Thus $X$ is unibranch at the origin. Let us w.l.o.g~assume $f_y(0)\neq 0$. Let $\gamma(t)$ be a parametrization of $X$ at the origin, i.e., $\gamma(t)=(x(t),y(t))\in\mathbb{C}\ps{t}^2$ a pair of power series for which the ring map 
\begin{align*}
\gamma^*\colon\mathbb{C}\ps{x,y}/(f)&\rightarrow \mathbb{C}\ps{t}\\
x&\mapsto x(t)\\
y&\mapsto y(t)
\end{align*}
is injective and $\gamma(0)=0$. Differentiating now both sides of the equality $f(x(t),y(t))=0$ with respect to $t$  gives us 
\begin{align}\label{gamma}
f_x(x(t),y(t))\cdot x'(t)+f_y(x(t),y(t))\cdot y'(t)=0.
\end{align}
Notice, that from $f_y(0)\neq 0$ it follows that $f_y(\gamma(t))\neq 0$. From equality (\ref{gamma}), we immediately see 
\begin{itemize}
\item[(1)]
$\displaystyle \kappa_0(\underline{x(t)},\underline{y(t)})=-\frac{f_x}{f_y}(x(t),y(t))$,
\item[(2)]
$\displaystyle \kappa_1(\underline{x(t)},\underline{y(t)})=-\frac{f_{xx}f_y^2-2f_{xy}f_xf_y+f_{yy}f_x^2}{f_y^3}(x(t),y(t))$,
\item[(3)] 
$\displaystyle \kappa_i(\underline{x(t)},\underline{y(t)})=\frac{\partial_x\kappa_{i-1}(f)\cdot f_y-\partial_y\kappa_{i-1}(f)\cdot f_x}{f_y}(x(t),y(t)),\ \text{ for }i\geq 2$.
\end{itemize}
Hence, we set
\begin{align*}
&\widetilde{\kappa}_0\coloneqq-\frac{f_x}{f_y},\\
&\widetilde{\kappa}_1\coloneqq-\frac{f_{xx}f_y^2-2f_{xy}f_xf_y+f_{yy}f_x^2}{f_y^3},\\
&\widetilde{\kappa}_i\coloneqq \frac{\partial_x\kappa_{i-1}(f)\cdot f_y-\partial_y\kappa_{i-1}(f) \cdot f_x}{f_y},\ \text{ for } i\geq 2,
\end{align*}
to be the implicit expressions of the algebraic curvatures. Together with the fact that the field of geometric invariants of plane curves is generated over $\mathbb{C}$ by $x^{(0)},y^{(0)}$, the slope and the (first and the higher) algebraic curvatures, we obtain the following theorem:

\begin{theorem}\label{impl}
Given $p=\frac{p_1}{p_2}\in I_F$ a geometric invariant of plane curves, there exist polynomials $\tilde{p}_1,\tilde{p}_2$ in $f$ and its partial derivatives, i.e., 
\[
\tilde{p}_1,\tilde{p}_2\in\mathbb{C}[x,y]
\] 
such that 
\[
p(\underline{x(t)},\underline{y(t)})=\frac{\tilde{p}_1(x(t),y(t))}{\tilde{p}_2(x(t),y(t))}
\]
for all parametrizations $(x(t),y(t))$ of $X$ satisfying $p_2(\underline{x(t)},\underline{y(t)})\neq 0$. In other words, each geometric invariant (of a given plane curve) admits an implicit description.\\
Moreover, if $p\in I_F\setminus\mathbb{C}(x^{(0)},y^{(0)})$, then we have even  
\[
\tilde{p}_1,\tilde{p}_2\in\mathbb{C}[\partial^i_x\partial_y^jf\,:\,i,j\in\mathbb{N}]\subseteq\mathbb{C}[x,y].
\]
\end{theorem} 

Finally, we prove that we are able to reconstruct the analytic branch of $X$ at the origin from the values of the implicit expressions of the slope and of algebraic curvatures of plane curves $\widetilde{\kappa}_i(x,y)$ at $(0,0)$.
 
\begin{corollary}\label{reconstr}
Let us assume that $\widetilde{\kappa}_i(0)<\infty$ for all $i\in\mathbb{N}$. Then the equation
\[
y-\sum_{i\geq 0}\frac{\widetilde{\kappa}_j(0)}{(i+1)!}\cdot x^{i+1}=0
\]
defines the analytic branch of $X$ at the origin.
\end{corollary}

\begin{proof}
Notice first that by the assumption $\widetilde{\kappa}_i(0)<\infty$, the case $f_y(0)\neq 0$ is excluded. Hence, the Implicit Function Theorem applies to $f$ and guarantees the existence of a parametrization of the form $(t,y(t))\in\mathbb{C}\ps{t}^2$ of $X$ at $0$, where $y(t)$ is even a convergent power series (since it is algebraic). The analytic branch of $X$ at $0$ is thus defined by the equation $g=y-y(x)$. The power series $y(t)$ can be expressed by the Taylor expansions as 
\[
y(t)=\sum_{i\geq 1}\frac{\partial_t^iy(0)}{i!}\cdot t^i=\sum_{i\geq 0}\frac{\kappa_{i}(\underline{t},\underline{y(t)})|_{t=0}}{(i+1)!}\cdot t^{i+1}.
\]
After rewriting each $\kappa_i(\underline{t},\underline{y}(t))$ as $\widetilde{\kappa}_i(t,y(t))$ and substituting $t=0$, we obtain the claimed equality.
\end{proof}

\begin{remark}\normalfont
Notice that for each plane algebraic curve $X\subseteq\mathbb{A}^2_{\mathbb{C}}$ defined by a square-free polynomial $f$ that is smooth at the origin, either $f_x(0)\neq 0$ or $f_y(0)\neq 0$ holds. Therefore, in the case of $f_y(0)=0$, we can just use the coordinate change $x\mapsto y, y\mapsto x$, in order to reach the assumptions of Corollary \ref{reconstr}. In such case, the corollary gives us for the analytic branch of $X$ at the origin an analytic equation of the form $x-x(y)=0$.
\end{remark}

\section{Geometric Invariants of Space Curves}

As already mentioned, the concept of geometric invariants and the ideas and techniques used in the case of plane curves can be easily extended to space curves in $\mathbb{A}_\mathbb{C}^{n+1}$ as well. There are only few technicalities we have to deal with and which have to be carried out explicitly. Hence, as the proofs of the most results about geometric invariants of space curves follow the same punch line as in the plane curve case, we will not repeat them completely. Instead of that we will often refer to the corresponding statements and proofs from the previous section and will rather concentrate on fixing new difficulties which appear when considering higher embedding dimensions $n+1$ with $n\geq 2$.

For each $n\in\mathbb{N}, n\geq 1$ let us consider the set of variables $x^{(i)},y_j^{(i)}$ for $i,j\in\mathbb{N}, 1\leq j\leq n$ and the field 
\[
F_n\coloneqq\mathbb{C}(x^{(i)},y_j^{(i)}) \,:\,i,j\in\mathbb{N}, 1\leq j\leq n).
\]
The integer $n+1$ stands for the embedding dimension of the space curves. We extend the derivation $\partial$ to $F_n$ by $\partial(y_j^{(i)})=y_j^{(i+1)}$ and thus, obtain the differential field $(F_n,\partial)$. Let us define
\[
L_n\coloneqq F_n(\varphi^{(i)},i\in\mathbb{N})=\mathbb{C}(x^{(i)},y_j^{(i)},\varphi^{(i)}  \,:\, i,j\in\mathbb{N}, 1\leq j\leq n).
\] 
Further, we extend the symbolic chain rule $\chi$ and the field homomorphism $\Lambda$ to $L_n$ by
\[
\chi(y_j^{(i)})=y_j^{(i+1)}\varphi^{(1)} \ \text{ and }\ \Lambda(y_j^{(i)})=\chi^{i}(y_j^{(0)}),
\]
respectively and define geometric invariants of space curves as those rational expressions that are invariant under $\Lambda$.

\begin{definition}\normalfont
We call a rational expression $p(\underline{x},\underline{y_j}  \,:\, 1\leq j\leq n)$ in $x^{(0)},y_j^{(0)}$ and their higher symbolic derivatives a \emph{geometric invariant of algebraic space curves (of embedding dimension $n+1$)} if it stays fixed under $\Lambda$, i.e., if the following equality is fulfilled
\[
p=\Lambda(p).
\]
\end{definition}

By $I_{F_n}$ we denote the corresponding invariant field, the field of all geometric invariants of space curves of embedding dimension $n+1$. 

We follow now the following strategy to construct geometric invariants of space curves: We use geometric invariants of plane curves, which we have already studied and whose basic properties are already known to us. Notice that each geometric invariant of plane curves, except for polynomials in $x^{(0)}$ with coefficients in $\mathbb{C}$, gives rise to $n$ different geometric invariants of space curves of embedding dimension $n+1$: 
Each geometric invariant $p$ of plane curves is a rational function in variables $x^{(i)},y^{(i)},i\in\mathbb{N}$. Let us emphasize the set of variables by writing $p(\underline{x},\underline{y})$ instead of just $p$. Hence, replacing each variable $y^{(i)}$ in $p(\underline{x},\underline{y})$ by $y_j^{(i)}$, for some $j$, does not disturb the invariance and yields therefore a geometric invariant of algebraic space curves. Using this substitution we define
\[
\kappa_{i,j}:=\kappa_i(\underline{x},\underline{y_j}),
\]
and call these expressions again the \emph{slopes (of the tangent vector)} in the case $i=0$ and (the \emph{first} if $i=1$ and the \emph{higher} for $i\geq 2$) \emph{algebraic curvatures} (\emph{of space curves}) otherwise. In this way we obtain the following system of geometric invariants of space curves:
\[
x^{(0)},y_j^{(0)},\kappa_{i,j},\quad \text{ where } i,j\in\mathbb{N} \text{ and } 1\leq j \leq n.
\] 
It turns even out that they represent a complete system of generators of the field of geometric invariants of space curves of embedding dimension $n+1$ and that each geometric invariant of space curves can be uniquely written as a rational function in the algebraic curvatures when applying the following $\mathbb{C}$-morphism to it
\begin{align*}
&i_\kappa:F_n\rightarrow F_n\\
&x^{(0)}\mapsto x^{(0)},x^{(1)}\mapsto 1, x^{(i)}\mapsto 0\ \text{ for all }\ i\geq 2\\
&y_j^{(0)}\mapsto y_j^{(0)},y_j^{(i)}\mapsto \kappa_{i-1,j}\ \text{ for all }\ 1\leq i,1\leq j\leq n.
\end{align*}
More precisely:

\begin{theorem}\label{14}
The field of geometric invariants of space curves of embedding dimension $n+1$ is generated over $\mathbb{C}$ by the variables $x^{(0)},y_j^{(0)}$, the slopes and the first and higher algebraic curvatures $\kappa_{i,j},i,j\in\mathbb{N},1\leq j\leq n,$ i.e.,
\begin{align}\label{inv}
I_{F_n}=\mathbb{C}(x^{(0)},y_j^{(0)},\kappa_{i,j}  \,:\, i,j,\in\mathbb{N},1\leq j\leq n).
\end{align}
Moreover, for each geometric invariant $p(\underline{x},\underline{y_1}\dots,\underline{y_n})$ the following equality is fulfilled 
\begin{align}\label{invar}
p=i_{\kappa}(p).
\end{align}
\end{theorem}

There are again two proofs of the first part of this theorem. If we just replace the field $F$ by $F_n$ and $J=~\mathbb{C}(x^{(0)},y^{(0)},\kappa_i  \,:\, i\in\mathbb{N})$ by the field $J_n\coloneqq\mathbb{C}(x^{(0)},y_j^{(0)},\kappa_{i,j}:i,j\in\mathbb{N},1\leq j\leq n)$ in Lemma \ref{15}, Proposition \ref{17} and the first proof of Theorem \ref{12}, we obtain already one proof of the equality (\ref{inv}). However, this proof does not explain the second part of Theorem \ref{14}, namely the equality (\ref{invar}). For this we need again the trick from the second proof of Theorem \ref{12} presented in the previous section.

\begin{proof}
The proof follows the same line as the second proof of Theorem \ref{12}. We consider again $x(t),y_1(t),\dots,y_n(t)\in \mathbb{C}\ps{t}$ a family of D-algebraically independent power series with the property that $\ord(x(t))=1.$  Using equation (\ref{1}), which obviously holds also for $(n+1)$-tuples of power series, the equality 
\[
p(\underline{x(t)},\underline{y_1(t)},\dots,\underline{y_n}(t))=i_{\kappa}(p)(\underline{x(t)},\underline{y_1(t)},\dots,\underline{y_n}(t))
\]
can be shown with the same trick as in the second proof of Theorem \ref{12}. Finally, we use the D-algebraic independence of the power series $x(t),y_1(t),\dots,y_n(t)$ and conclude the required equality in the field of geometric invariants of space curves.
\end{proof}

\vspace{3mm}

\centerline{\textbf{Implicit Expressions of Geometric Invariants of Space Curves}}

\vspace{3mm}

Similarly to the plane curve case, geometric invariants of space curves do admit implicit expressions in terms of the defining implicit equations of space curves and their higher derivatives as well. To find these implicit expressions we differentiate again the composition of the implicit equations with a parametrization and use the chain rule.

More precisely, consider an algebraic space curve $X=V(I)\subseteq\mathbb{A}_\mathbb{C}^{n+1}$ defined by a radical ideal $I$. Further assume $0$ being a smooth point on $X$. Let $\gamma(t)$ be a parametrization of $X$ at $0$. Here, by a parametrization we again mean an $(n+1)$-tuple of univariate power series $\gamma(t)=~(x(t),y_1(t),\dots,y_n(t))\in\mathbb{C}\ps{t}^{n+1}$ for which the corresponding ring map
\begin{align*}
\gamma^*\colon \mathbb{C}\ps{x,y_1,\dots,y_n}/I&\rightarrow \mathbb{C}\ps{t}\\
x&\mapsto x(t)\\
y_j&\mapsto y_j(t), \ \text{ for } 1\leq j\leq n
\end{align*}
is injective. Let us write $I=(f_1,\dots,f_r)$, with $ f_j\in~\mathbb{C}[x,y_1,\dots,y_{n},]$ the generators of $I$ for some $r\geq n$. If we differentiate the equalities $f_j(\gamma(t))=0$ for all $1\leq j\leq r$ with respect to $t$, we get the following system of equations:
\begin{align}\label{2} J_{f_1,\dots,f_r}(\gamma(t))\cdot\gamma'(t)=0,
\end{align}
where $J_{f_1,\dots,f_r}$ denotes the Jacobian matrix of $f_1,\dots,f_r$. The strategy how to find implicit expressions for the slopes $\kappa_{0,j}$, is to eliminate all components of the vector $\gamma'(t)$ except for $x'(t)$ and $y_j'(t)$ in equation (\ref{2}), or, in other words, to express each $y'_j(t)$ as a linear function in the parameter $x'(t)$. For this purpose we need only $n$ rows of the Jacobian matrix $J_{f_1,\dots,f_r}$ that are linearly independent, let us say the first $n$ rows.

\begin{remark}\normalfont
Recall that for any point $a\in X$ we have $\rk(J_{f_1,\dots,f_r}(a))\leq \codim(X)=n$ and the equality holds if and only if $a$ is a smooth point of $X$. Hence, as $0$ is a smooth point of $X$, if it happens that the first $n$ rows of the Jacobian matrix $J_{f_1,\dots,f_r}$ are linearly dependent at $0$, we can always reorder the generators of the ideal $I$, let us say $f_{\sigma(1)},\dots,f_{\sigma(r)}$ for some permutation $\sigma\in S_r$, so that $J_{f_{\sigma(1)},\dots,f_{\sigma(n)}}$ is invertible at $0$. Thus, we can w.l.o.g.~always assume that $\det(J_{f_1,\dots,f_n}(0))\neq 0$ and so $\det\bigl(J_{f_1,\dots,f_n}(\gamma(t))\bigr)\neq 0.$ 
\end{remark}

Let us rewrite the first $n$ rows of the equation (\ref{2}) into
\begin{align*}
{\cal J}(\gamma(t))\cdot\left(\frac{y'_j(t)}{x'(t)}\right)_{j=1}^n=
-\bigl(\partial_{x}f_j(\gamma(t))\bigr)_{j=1}^n,
\end{align*}
with
\[
{\cal J}\coloneqq\begin{pmatrix} \partial_{y_{1}}f_1 & \cdots & \partial_{y_{n}}f_1\\ 
\vdots & \ddots & \vdots\\
\partial_{y_{1}}f_n & \cdots & \partial_{y_{n}}f_n  \end{pmatrix}.
\]
Applying the Cramer's rule to the above system of linear equations yields for each $j=1,\dots,n$ the equality
\[
\kappa_{0,j}(\underline{x(t)},\underline{y_j(t)})=\frac{y_j'(t)}{x'(t)}=\frac{\det{\cal J}_j(\gamma(t))}{\det{\cal J}(\gamma(t))},
\]
where ${\cal J}_j$ is the matrix formed by replacing the $j$-th column of ${\cal J}$ by the column vector $-(\partial_xf_j)_{j=1}^n$, i.e.,

\[
{\cal J}_j\coloneqq\begin{pmatrix} \partial_{y_{1}}f_1 &  \cdots & \partial_{y_{j-1}}f_1 & -\partial_{x}f_1&  \partial_{y_{j+1}}f_1& \cdots & \partial_{y_{n}}f_1\\ 
\vdots & \ddots & \vdots & \vdots &\vdots & \ddots & \vdots\\
\partial_{y_{1}}f_n&  \cdots & \partial_{y_{j-1}}f_n & -\partial_{x}f_n&  \partial_{y_{j+1}}f_n& \cdots & \partial_{y_{n}}f_n\end{pmatrix}.
\]
We set
\[
\widetilde{\kappa}_{0,j}\coloneqq\frac{\det{\cal J}_j}{\det{\cal J}}
\]
to be an implicit expression for the slope $\kappa_{0,j}$.

Let us assume that we have already computed $\widetilde{\kappa}_{i,j}$, an implicit expression for~$\kappa_{i,j}$. Then, for the higher algebraic curvatures $\kappa_{i,j},i\geq 1$ we have by definition
\begin{align*}
\kappa_{i+1,j}(\underline{x(t)},\underline{y_j(t)}) =\ &\frac{1}{x'(t)}\cdot\partial_t\bigl(\kappa_{i,j}(\underline{x(t)},\underline{y_j(t)})\bigr) =\ \frac{1}{x'(t)}\cdot\partial_t\bigl(\widetilde{\kappa}_{i,j}(\gamma(t))\bigr)\\
=\ &\frac{1}{x'(t)}\left(\partial_{x}\widetilde{\kappa}_{i,j}(\gamma(t))\cdot x'(t)+\sum_{k=1}^n\partial_{y_k}\widetilde{\kappa}_{i,j}(\gamma(t))\cdot y'_k(t) \right)\\
=\ &\partial_{x}\widetilde{\kappa}_{i,j}(\gamma(t))+\sum_{k=1}^n\partial_{y_k}\widetilde{\kappa}_{i,j}(\gamma(t))\cdot\widetilde{\kappa}_{0,k}(\gamma(t)).
\end{align*}
We set
\[
\widetilde{\kappa}_{i+1,j}\coloneqq \partial_{x}\widetilde{\kappa}_{i,j}+\sum_{k=1}^n\partial_{y_k}\widetilde{\kappa}_{i,j}\cdot\widetilde{\kappa}_{0,k}
\]
and call it an implicit expression for the algebraic curvature $\kappa_{i+1,j}$. 

Now, since the algebraic curvatures generate the whole field of geometric invariants of space curves, we conclude the existence of an implicit expression for an arbitrary geometric invariant of space curves. 

Let us further for $p$ a geometric invariant denote the evaluation $p(\underline{x(t)},\underline{y_1(t)},\dots,\underline{y_n(t)})$ by $p(\underline{\gamma(t)})$.

\begin{theorem}\label{impl:sp}
For each geometric invariant $p=\frac{p_1}{p_2}$ of space curves of embedding dimension $n+1$ there exist polynomials 
\[
\tilde{p}_1,\tilde{p}_2\in\mathbb{C}[x,y_1,\dots,y_n],
\]
such that the equality
\[
p(\underline{\gamma(t)})=\frac{\tilde{p}_1(\gamma(t))}{\tilde{p}_2(\gamma(t))}
\]
is satisfied for all parametrizations $\gamma(t)$ of $X$ for which $p_2(\underline{\gamma(t)})\neq 0$. In other words, each geometric invariant (of a given space curve) admits an implicit description (in terms of its defining equations).\\
Moreover, if $p\in I_{F_n}\setminus\mathbb{C}(x^{(0)},y_j^{(0)}\,:\,1\leq j\leq n)$, then the polynomials $\tilde{p}_1$ and $\tilde{p}_2$ are even polynomials in the generators $f_1,\dots,f_r$ and their partial derivatives, i.e.,
\[
\tilde{p}_1,\tilde{p}_2\in\mathbb{C}[\partial_{x}^{i_0}\partial_{y_1}^{i_1}\cdots\partial_{y_n}^{i_n}f_k\,:\,i_0,\dots,i_n\in\mathbb{N},1\leq k\leq r]\subseteq\mathbb{C}[x,y_1,\dots,y_n].
\]
\end{theorem}

As in the case of plane curves, the values of the slope and (higher) algebraic curvatures of a given space curve at a smooth point describe the curve completely. More precisely:

\begin{corollary}\label{sp:cur}
Suppose that $\widetilde{\kappa}_{i,j}(0)<\infty$ for all $i\in\mathbb{N},1\leq j\leq n$. Then the analytic branch of $X$ at $0$ is defined by the ideal
\[
\mathcal{I}_X=\left(y_j-\sum_{i\geq 0}\frac{\widetilde{\kappa}_{i,j}(0)}{(i+1)!}x^{i+1}  \,:\, 1\leq j\leq n\right).
\]
\end{corollary}

\begin{proof}
Since $0$ is a smooth point on $X$, the curve $X$ is locally at $0$ biholomorphic to an open subset of $\mathbb{C}$ containing $0$. Thus, $X$ can be parametrized at $0$ by a parametrization $\gamma(t)=(x(t),y_1(t),\dots,y_n(t))$ with at least one component of order equal to one. Further, the first $n$ rows of equality (\ref{2}) can be written as
\begin{align}\label{eval0}
y'_j(t)\cdot\det{\cal J}(\gamma(t))=x'(t)\cdot\det{\cal J}_j(\gamma(t)).
\end{align}
From the assumption $\widetilde{\kappa}_{i,j}(0)<\infty$ it follows that $\det{\cal J}(0)\neq 0$ and that $x'(0)\neq 0.$ Therefore, $\gamma(t)$ has the form
\[
\gamma(t)=(t,y_1(t),\dots,y_n(t))\in\mathbb{C}\{t\}^{n+1}.
\]
Once we have determined the components $y_j(t),j=1,\dots,n$, of the parametrization, we conclude immediately that the analytic branch of $X$ at $0$ is contained in the analytic variety defined by the ideal $(y_j-y_j(x):1\leq j\leq n)$. But this ideal is a prime ideal and its height equals~$n$. Hence, it defines already the analytic branch of $X$ at $0$. Further, the Taylor expansion yields
\[
y_j(t)=\sum_{i\geq 0}\frac{\partial_t^iy_j(0)}{i!}\cdot t^{i}=\sum_{i\geq 0}\frac{\kappa_{i,j}(\underline{t},\underline{y_j(t)})|_{t=0}}{(i+1)!}\cdot t^{i+1}
\]
for all $1\leq j\leq n$. Rewriting each $\kappa_{i,j}(\underline{t},\underline{y_j(t)}$ as $\widetilde{\kappa}_{i,j}(\gamma(t))$ and substituting $t=0$ finishes now the proof.
\end{proof}

\appendix
\section{Differentially Algebraic and Differentially Transcendental Power Series and their Families}\label{appendix}

We list some important and for the classification of geometric invariants essential facts about differentially algebraic and differentially transcendental power series in this section. The theorems and their proofs that are presented here, are classical results and techniques by J.~F.~Ritt and E.~Gourin \cite{gr27}, which can be found also in Rubel's survey~\cite{ru89}.

Recall that a univariate power series $f\in\mathbb{C}\ps{x}$ is called \emph{algebraic} if it is a solution of $p(x,f(x))=0$ for some bivariate nonzero polynomial $p\in\mathbb{C}[x,y]$, and that $f$ is called \emph{transcendental} otherwise. But it may happen that, even though $f$ is transcendental, it satisfies an algebraic differential equation, strictly speaking a partial differential equation, i.e., 
\begin{equation}\label{02}
q(x,\partial^{i}_{x}f  \,:\, 0\leq i\leq k-1)=0
\end{equation}
is fulfilled by $f$ for some nonzero polynomial $q\in\mathbb{C}[x,y_1,\dots,y_{k}]$ in $k+1$ variables, with $k\in\mathbb{N}$ some positive integer. In this case we say that $f$ is \emph{differentially algebraic} or \emph{D-algebraic} and otherwise we call $f$ \emph{transcendentally transcendental} or \emph{hypertranscendental} or \emph{D-transcendental}. The definition is due to E.~Kolchin \cite[Chapter I, \S 6]{ko73} and this concept appears for example also in works by L.~A.~Rubel \cite{ru92}, C.~Hardouin and M.~F.~Singer \cite{hs08}, T.~Dreyfus and  C.~Hardouin \cite{dh19} and A.~Ostrowski \cite[\S6]{os20}.

One can easily find examples of D-algebraic power series. However, to construct a D-transcendental power series is more tricky and requires a deeper understanding of the concept and of the basic properties of D-transcendental power series. We prove in this section the existence of D-transcendental power series (and families of power series).

Before we come to the proof, let us recall the notion of the
\emph{Wronskian determinant}. The \emph{Wronskian matrix} of the family $g_1,\dots,g_m\in\mathbb{C}\ps{x}$ of $m$ univariate power series is defined as
\[
\begin{pmatrix}
g_1 & \cdots & g_m\\
\partial_x g_1 & \cdots & \partial_x g_m\\
\vdots & \vdots & \vdots\\
\partial_x^{m-1}g_1 &\cdots & \partial_x^{m-1}g_m
\end{pmatrix}.
\]
The determinant of this matrix is called the \emph{Wronskian determinant} of this family. Obviously, the Wronskian determinant of a linearly dependent family of power series equals zero. It can be shown (see e.g. \cite[pp.~90-92]{bo01}) that the converse is true as well. So we have the following result:

\begin{lemma}\label{wronskian}
A family of finitely many power series is linearly independent over $\mathbb{C}$ if and only if its Wronskian determinant equals zero.
\end{lemma}

Now we use Lemma \ref{wronskian} in order to prove the following statement:

\begin{theorem}\label{01}
Let $f\in\mathbb{C}\ps{x}$ be a D-algebraic power series. Then $f$ satisfies an algebraic differential equation with integer coefficients.
\end{theorem}

The proof of Theorem \ref{01} presented here is due to E.~Gourin and J.~F.~Ritt \cite[\S 2]{gr27}. Before proving the theorem we introduce the following technical lemma:

\begin{lemma}\label{w}
For any number of points in $\mathbb{C}$, we can always construct a polynomial $f\in\mathbb{C}[x]$ with any given values of itself and its first $k-1$ derivatives at these points. More precisely: Let $k,N\in\mathbb{N}$ be two positive integers. Consider arbitrary $N$ points $a_r \in\mathbb{C}$ and $N$ vectors $b_r=(b_{r,0},\dots,b_{r,k-1})\in\mathbb{C}^{k}$ for $1\leq r \leq N$. Then there exists a polynomial $f\in\mathbb{C}[x]$ satisfying the equality
\begin{align}\label{eq:w}
\bigl(f(a_r)\,,\,\partial_xf(a_r)\,,\,\dots\,,\,\partial_x^{k-1}f(a_r)\bigr)=b_{r},
\end{align}
for each $1\leq r\leq N$.
\end{lemma}

\begin{proof}
We prove the claim by induction on $N$. The induction base follows immediately by setting
\[
f:=\sum_{i=0}^{k-1}\frac{b_{1,i}}{i!}(x-a_1)^i.
\]
Let now $g\in\mathbb{C}[x]$ be a polynomial satisfying 
\[
\bigl(g(a_r),\partial_xg(a_r),\dots,\partial_x^{k-1}g(a_r)\bigr)=b_{r},
\]
for all $1\leq r\leq N-1$. We then define 
\[
f\coloneqq g+\underbrace{\prod_{r=1}^{N-1}(x-a_r)^{k}\cdot\sum_{i=0}^{k-1}\frac{\tilde{b}_i}{i!}(x-a_N)^i}_{\eqqcolon\tilde{g}},
\]
where the $\tilde{b}_i$'s can be iteratively computed as the solution of the following system of linear equations:
\begin{align*}
b_{N,i}&=\partial_x^i g(a_N)+\partial_x^i\tilde{g}(a_N)\\
&=\partial_x^i g(a_N)+\sum_{j=0}^i\tilde{b}_j\binom{i}{j}\partial_x^{i-j}\left(\prod_{r=1}^{N-1}(x-a_r)^{k}\right)|_{x=a_N},
\end{align*}
for $0\leq i\leq k-1$. Notice that each $\tilde{b}_i$ is uniquely given by a linear combination of $\tilde{b}_j$'s with $j<i$. It follows then that the polynomial $f$ constructed in this way satisfies condition (\ref{eq:w}).
\end{proof}

\begin{proof}[Proof of Theorem \ref{01}]
Any algebraic differential equation of the form (\ref{02}) satisfied by a power series $f$ can be written as
\begin{equation}\label{03}
\sum_{j\in J,l\in L}c_{(j,l)}x^j\underline{f}^l=0,
\end{equation}
where $J\subseteq\mathbb{N}$ and $ L\subseteq\mathbb{N}^{k}$, for some $k\in\mathbb{N}$, and 
\[
\underline{f}^l=\prod_{1\leq i\leq k} (\partial_x^{i-1}f)^{l_i}.
\]
By construction, all the expressions $x^j\underline{f}^l$ are distinct from each other. This means that the family of power series $x^j\underline{f}^l$ is linearly dependent over $\mathbb{C}$ and thus its Wronskian determinant must equal zero. As this Wronskian determinant is a polynomial in $\partial_x^s(x^j\underline{f}^l)$ with integer coefficients, it is also a polynomial in $x$ and $f$ and its higher derivatives with integer coefficients. Therefore, it is enough to show that the polynomial $W(x,y_0,\dots,y_{k+m-2})\in\mathbb{Z}[x,y_0,\dots,y_{k+m-2}]$, with $m=|J|\cdot|L|$, defining the Wronskian determinant of the family $x^j\underline{f}^l$ is a non-zero polynomial to obtain an algebraic differential equation with integer coefficients as claimed. To be more precise, let us extend the derivation $\partial_x$ to the the polynomial ring $\mathbb{C}[x,y_0,\dots,y_{k+m-2}]$ by setting $\partial_x(y_i)\coloneqq y_{i+1}.$ Then the polynomial
\[
W(x,y_0,\dots,y_{k+m-2})=\det \bigl(\partial_x^s(x^j \prod_{1\leq i\leq k}y_{i-1}^{l_i})\bigr)_{\substack{0\leq s\leq m-1\\j\in J, \ l\in L}}
\]
defines the Wronskian determinant as it satisfies 
\[
W(x,f,\dots,\partial_x^{k+m-2}f)= \text{ the Wronskian determinant of the family } x^j\underline{f}^l.
\]
We therefore aim to show that $W\neq~0$. Let us assume by contradiction that $W=0$. We then have $W(x,g,\dots,\partial^{k+m-2}_xg)=0$ for any power series $g\in\mathbb{C}\ps{x}$. According to Lemma \ref{wronskian}, this would mean that for any power series $g$, the expressions $x^j\underline{g}^l$'s are linearly dependent over $\mathbb{C}$ and so they satisfy an equation like (\ref{03}), let us say 
\begin{align}\label{eq:a}
\sum_{j\in J,l\in L} a_{(j,l)}x^j\underline{g}^l(x)=0.
\end{align}
According to Lemma \ref{w}, for $N\in\mathbb{N}$ sufficiently large ($N\geq m$), we can always construct a polynomial $g\in\mathbb{C}[x]$ with the property that for $j\in J$ and $l\in L$, the vectors 
\[
\begin{pmatrix} 
1^j\underline{g}^l(1)\\
\vdots\\
N^j\underline{g}^l(N)
\end{pmatrix} 
\]
are linearly independent over $\mathbb{C}$. This, however, contradicts the linear relation (\ref{eq:a}).
\end{proof}

The concept of D-algebraicity can be extended also to families of power series. We call a family $f_1,\dots,f_l\subseteq\mathbb{C}\ps{x}$ of univariate power series \emph{differentially algebraically dependent} or \emph{D-algebraically dependent} if there exists a polynomial $q\in\mathbb{C}[x,y_1,\dots,y_{kl}]$ in $kl+1$ variables, for some $k\in\mathbb{N}$, such that 
\[
q(x,\partial_x^if_j  \,:\, 0\leq i\leq k-1,1\leq j\leq l)=0.
\]
Otherwise, we call the family \emph{differentially algebraically independent} or \emph{D-algebraically independent}. 

Let us remark that the concept of D-algebraically dependent families of power series was used for example by E.~Kolchin in \cite[Chapter II, \S 7]{ko73} or by C.~Hardouin and M.~F.~Singer in \cite{hs08}.

Notice first that, similarly to the case of D-algebraic power series (and it can be proven using the same argument as in the proof of Theorem \ref{01}), a family of D-algebraically dependent power series satisfies an algebraic differential equation with integer coefficients.

\begin{theorem}\label{06}
Let $f_1,\dots,f_l\in\mathbb{C}\ps{x}$ be a D-algebraically dependent family of power series. Then the power series $f_1,\dots,f_l$ satisfy an algebraic differential equation with integer coefficients.
\end{theorem}

The question now is whether it is always possible to construct arbitrarily large D-algebraically independent families of power series. The answer is ``yes'' as stated in the following theorem:

\begin{theorem}\label{09}
For any given positive integer $l\in \mathbb{N}$, there exists a  D-algebraically independent family of $l$ power series.
\end{theorem}

\begin{proof}
Assume by contradiction that any family $f_1,\dots,f_l\in\mathbb{C}\ps{x}$ of $l$ power series is D-algebraically dependent, i.e., the power series satisfy an algebraic differential equation
\begin{equation}\label{08}
q(x,\partial_x^if_j  \,:\, 0\leq i\leq k-1,1\leq j\leq l)=0,
\end{equation}
for some $k\in\mathbb{N}$. Then by Theorem \ref{06}, we may w.l.o.g. assume that $q$ has integer coefficients. Let $\alpha_0,\alpha_1,\alpha_2,\dots$ be any sequence of complex numbers that has infinite transcendence degree over $\mathbb{Q}$. Further, let $\Phi:\{1,\dots,l\}\times\mathbb{N}\rightarrow\mathbb{N}\backslash\{0\}$ be a bijection between $\{1,\dots,l\}\times\mathbb{N}$ and $\mathbb{N}\backslash\{0\}$. For each $j=1,\dots,l$ we define
\[
f_j\coloneqq\sum_{i\in\mathbb{N}}\frac{\alpha_{\Phi(j,i)}}{i!}(x-\alpha_0)^i.
\]
Hence, each derivative of $f_j$ satisfies $\partial_x^if_j(\alpha_0)=\alpha_{\Phi(j,i)}$. But substituting $x\mapsto\alpha_{0}$ into the equation (\ref{08}) yields an algebraic equation over $\mathbb{Q}$ for $\alpha_j$'s which is a contradiction to the infinite transcendence degree over $\mathbb{Q}$ of the sequence $\alpha_0,\alpha_1,\alpha_2,\dots$.
\end{proof}

\medskip\noindent {\bf Acknowledgements:} I am grateful to Herwig Hauser for sharing with me his ideas which gave rise to the main objectives presented in this note. I thank also to Josef Schicho for his methodological contributions to the technical parts of this note. Furthermore, I would like to express my sincere gratitude to Alin Bostan for his time, ideas and for the great introduction to the field of D-algebraic power series he provided to me. Finally, my sincere thanks also goes to Matteo Gallet for his valuable discussions and corrections of the first version of this text.

\end{document}